\documentclass[11pt,a4paper]{amsart}

\usepackage{graphicx,amssymb}
\usepackage[top=40mm,bottom=35mm,left=30mm,right=30mm]{geometry}
\usepackage[frame]{xy}
\xyoption{arc}
\xyoption{graph}
\xyoption{matrix}

\theoremstyle{plain}
\newtheorem{theorem}{Theorem}[section]
\newtheorem{lemma}[theorem]{Lemma}

\newtheorem{proposition}[theorem]{Proposition}

\theoremstyle{definition}
\newtheorem{definition}[theorem]{Definition}
\newtheorem{example}[theorem]{Example}

\makeatletter
\newif\if@my@Hidebox \@my@Hideboxtrue
\newbox\@my@Hidebox
\def\myHidebox{\setbox\@my@Hidebox\vbox\bgroup}
\def\endmyHidebox{\egroup
   \if@my@Hidebox
      \unvbox\@my@Hidebox\par
   \else\par
   \fi
}
\def\showmyHidebox{\@my@Hideboxtrue}
\def\hidemyHidebox{\@my@Hideboxfalse}
\def\showmyproof{\@my@Hideboxtrue}
\def\hidemyproof{\@my@Hideboxfalse}

\makeatother

\hidemyHidebox  
\showmyHidebox  

\let\le\leqslant
\let\ge\geqslant
\def\Z{\mathbb Z}

\def\St{\operatorname{St}}
\def\supp{\operatorname{supp}}
\def\diam{\operatorname{diam}}
\def\minTL{\mathcal L}

\begin{document}

\title
[Minimal asymptotic translation length of right-angled Artin groups]
{An upper bound of the minimal asymptotic translation length
of right-angled Artin groups \\
on extension graphs}

\author{Eon-Kyung Lee and Sang-Jin Lee}
\address{Department of Mathematics and Statistics, Sejong University, Seoul, Korea}
\email{eonkyung@sejong.ac.kr}

\address{Department of Mathematics, Konkuk University, Seoul, Korea}
\email{sangjin@konkuk.ac.kr}
\date{\today}

\begin{abstract}
For the right-angled Artin group action on the extension graph,
it is known that the minimal asymptotic translation length is
bounded above by 2 provided that the defining graph has diameter at least 3.
In this paper, we show that the same result holds
without any assumption.
This is done by exploring some graph theoretic properties of
biconnected graphs, i.e.\ connected graphs whose complement is also connected.

\end{abstract}

\maketitle


\section{Introduction}

For a simplicial graph $\Gamma = ( V(\Gamma), E(\Gamma))$,
let $d_\Gamma(v_1,v_2)$ denote the graph metric on $V(\Gamma)$,
i.e.\ the length of the shortest path in $\Gamma$ from $v_1$ to $v_2$.
Let $\diam(\Gamma)$ denote the \emph{diameter} of $\Gamma$,
i.e.\ $\diam(\Gamma)=\max\{ d_\Gamma(v_1,v_2): v_1, v_2\in V(\Gamma)\}$.
Let $\bar\Gamma$ denote the \emph{complement graph} of $\Gamma$,
i.e.\ $V(\bar\Gamma)=V(\Gamma)$ and
$E(\bar\Gamma)=\{\{v_1,v_2\}\subset V(\Gamma): v_1\ne v_2,\
\{v_1,v_2\}\not\in E(\Gamma)\}$.
$\Gamma$ is called \emph{nontrivial} if $|V(\Gamma)|\ge 2$.
$\Gamma$ is called \emph{biconnected} if both
$\Gamma$ and $\bar\Gamma$ are connected.

The \emph{right-angled Artin group} $A(\Gamma)$
on a finite simplicial graph $\Gamma$ is the group generated by
the vertices of $\Gamma$ such that two generators commute if they are adjacent,
hence
$$
A(\Gamma)=\langle\, v\in V(\Gamma) : v_iv_j=v_jv_i\ \
\mbox{if $\{v_i,v_j\}\in E(\Gamma)$}\,\rangle.
$$

The \emph{extension graph} $\Gamma^e$ of $\Gamma$ is the graph such that
the vertex set $V(\Gamma^e)$ is the set of all elements of $A(\Gamma)$
that are conjugate to a vertex of $\Gamma$,
and two vertices $v_1^{g_1}$ and $v_2^{g_2}$
are adjacent in $\Gamma^e$ if and only if
they commute as elements of $A(\Gamma)$.
(Here, $v^g$ denotes the conjugate $g^{-1}vg$.)
Therefore
\begin{align*}
V(\Gamma^e) & =\{v^g: v\in V(\Gamma),\ g\in A(\Gamma)\},\\
E(\Gamma^e) & =\{\, \{ v_1^{g_1}, v_2^{g_2}\}\,:\,
v_1^{g_1}, v_2^{g_2} \in V(\Gamma^e),\
v_1^{g_1}v_2^{g_2}=v_2^{g_2}v_1^{g_1}\ \mbox{in $A(\Gamma)$}\,\}.
\end{align*}
Extension graphs are usually infinite and locally infinite.
They are very useful
in the study of right-angled Artin groups
such as the embeddability problem
between right-angled Artin groups~\cite{KK13, KK14a, LL16, LL18}.
Let $d_e$ denote the graph metric $d_{\Gamma^e}$.
It is known that $(\Gamma^e,d_e)$ is a quasi-tree,
and hence a hyperbolic graph~\cite{KK13}.

Throughout this paper, all graphs are simplicial, and assumed to be finite
except for extension graphs.

\begin{definition}
Suppose a group $G$ acts on a connected metric space $(X, d)$
by isometries from right.
The \emph{asymptotic translation length} of an element $g\in G$,
denoted $\tau_{(X,d)}(g)$ or $\tau(g)$,
is defined by
$$
\tau(g)= \lim_{n\to\infty} \frac{d(xg^n, x)}{n},
$$
where $x\in X$.
This limit always exists, is independent of the choice of $x$,
and satisfies
$\tau(g^n) = |n|\tau(g)$
and $\tau(h^{-1}gh) = \tau(g)$ for all $g, h\in G$ and $n\in\Z$.
If $\tau(g)>0$, $g$ is called \emph{loxodromic}.
If $\{d(xg^n, x)\}_{n=1}^\infty$ is bounded,
$g$ is called \emph{elliptic}.
If $\tau(g)=0$ and $\{d(xg^n, x)\}_{n=1}^\infty$ is unbounded,
$g$ is called \emph{parabolic}.
The \emph{minimal asymptotic translation length} of $G$
for the action on $(X,d)$, denoted $\minTL_{(X,d)}(G)$,
is defined by
$$
\minTL_{(X,d)}(G) = \min\{ \tau_{(X,d)}(g):
g\in G,~\tau_{(X,d)}(g)>0\}.
$$
\end{definition}

The (minimal) asymptotic translation lengths have been studied
for the actions of various geometric groups
(see \cite{LL07, GT11, KS19, BS20, Gen22, BSS23} for example).
For the action of $A(\Gamma)$ on $(\Gamma^e,d_e)$,
Baik, Seo and Shin showed the following
which is an analogue of the theorem of Bowditch~\cite[Theorem 1.4]{Bow08}.

\begin{theorem}[{\cite[Theorem A]{BSS23}}]
Let $\Gamma$ be a nontrivial biconnected graph.
For the action of $A(\Gamma)$ on $(\Gamma^e,d_e)$,
all loxodromic elements have rational asymptotic translation lengths.
If\/ $\Gamma$ has girth at least 6 in addition,
the asymptotic translation lengths have a common denominator.
\end{theorem}

As in the above theorem, it is natural to require that ``$\Gamma$ is a nontrivial biconnected graph''
when we consider the action of $A(\Gamma)$ on $\Gamma^e$,
because $\Gamma^e$ is connected with infinite diameter
if and only if $\Gamma$ is nontrivial and biconnected~\cite[Lemma 3.5]{KK13}.
Observe that if $|V(\Gamma)|\in\{2,3\}$
then either $\Gamma$ or $\bar\Gamma$ is disconnected.
Hence the statement that ``$\Gamma$ is a nontrivial biconnected graph''
implies $|V(\Gamma)|\ge 4$.

\smallskip

For the lower bound of the minimal asymptotic translation length
of $A(\Gamma)$ on $(\Gamma^e,d_e)$,
it follows from a result of Kim and Koberda~\cite[Lemma 33]{KK14b} that
$\minTL_{(\Gamma^e,\,d_e)}(A(\Gamma))
\ge \frac1{2|V(\Gamma)|^2}$.
This lower bound is improved to
$\minTL_{(\Gamma^e,\,d_e)}(A(\Gamma))
\ge \frac1{|V(\Gamma)|-2}$
in~\cite[Theorem 6.5]{LL22}.

For the upper bound, Baik, Seo and Shin showed the following.

\begin{theorem}[{\cite[Theorem F]{BSS23}}]
Let $\Gamma$ be a nontrivial biconnected graph.
If\/ $\diam(\Gamma)\ge 3$, then
$$\minTL_{(\Gamma^e,\,d_e)}(A(\Gamma))\le 2.$$
\end{theorem}

The graphs with diameter $2$ form a large class of graphs.
Therefore it would be natural to ask whether the same upper bound
holds for these graphs.
The following is the main result of this paper,
which shows that the above upper bound
holds without any assumption.

\medskip\noindent
\textbf{Theorem A} (Theorem~\ref{thm:tl})\ \
\em
For any nontrivial biconnected graph  $\Gamma$,
$$\minTL_{(\Gamma^e,\,d_e)}(A(\Gamma))\le 2.$$
\par\medskip\upshape

If $\diam(\bar\Gamma)$ is large,
we obtain a better upper bound as follows.

\medskip\noindent
\textbf{Theorem B} (Theorem~\ref{thm:d_TL})\ \
\em
Let $\Gamma$ be a biconnected graph.
If\/ $d=\diam(\bar\Gamma)\ge 3$, then
$$\minTL_{(\Gamma^e,\,d_e)}(A(\Gamma))\le \frac{2}{d-2}.$$
\par\medskip\upshape

Theorem B says that if $d=\diam(\bar\Gamma)$ is large,
then $\minTL_{(\Gamma^e,\,d_e)}(A(\Gamma))$ is small.
One may ask whether the converse holds, i.e.\
if $\minTL_{(\Gamma^e,\,d_e)}$ is small, then $\diam(\bar\Gamma)$ is large.
The following theorem shows that this is not the case.

\medskip\noindent
\textbf{Theorem C} (Theorem~\ref{thm:V_TL})\ \
For any $n\ge 4$, there exists a biconnected graph $\Gamma_n$
such that
$$|V(\Gamma_n)|=n+2, \quad \diam(\Gamma_n)=\diam(\overline{\Gamma_n})=2,
\quad
\minTL_{(\Gamma_n^e,\,d_e)}(A(\Gamma_n))\le \frac{2}{n-3}.$$
\par\medskip\upshape

\section{Preliminaries}

\subsection{Notations on graphs}

For $v\in V(\Gamma)$ and $A\subset V(\Gamma)$,
the \emph{stars} $\St_\Gamma(v)$ and $\St_\Gamma(A)$
are defined as
\begin{align*}
\St_\Gamma(v) &=\{u\in V(\Gamma):d_\Gamma(u,v)\le 1\},\\
\St_\Gamma(A) &=\bigcup\nolimits_{v\in A}\St_\Gamma(v)
    =\{u\in V(\Gamma):d_\Gamma(u,v)\le 1~\mbox{for some $v\in A$}\}.
\end{align*}
For an edge $e=\{v_1,v_2\}$,
the star $\St_\Gamma(\{v_1,v_2\})$ is also denoted by $\St_\Gamma(e)$.

\begin{definition}[dominate]
Let $A$ and $B$ be subsets of $V(\Gamma)$.
If $B\subset \St_\Gamma(A)$
(i.e.\ each vertex $v\in B$ is either contained in $A$ or
adjacent to a vertex in $A$),
then we say that $A$ \emph{dominates} $B$ in $\Gamma$.
If $A$ dominates $V(\Gamma)$ in $\Gamma$,
then we simply say that $A$ \emph{dominates} $\Gamma$.
\end{definition}
From the above definition, the statement that $A$ dominates $\bar\Gamma$ means $\St_{\bar\Gamma}(A)= V(\Gamma)$.

For $A \subset V(\Gamma)$,
$\Gamma[A]$ denotes the \emph{subgraph of $\Gamma$ induced by $A$}, i.e.\
$$
V(\Gamma[A])=A,\qquad
E(\Gamma[A])=\{\,\{v_1,v_2\} : v_1,v_2\in A,\ \{v_1,v_2\}\in E(\Gamma)\}.
$$

A graph $\Gamma_0$ is called an \emph{induced subgraph} of $\Gamma$
if $\Gamma_0=\Gamma[A]$ for some $A\subset V(\Gamma)$.
Notice the following.
\begin{itemize}
\item[(i)]
For $A\subset V(\Gamma)$, $\bar\Gamma[A]=\overline{\Gamma[A]}$,
i.e.\ the subgraph of $\bar\Gamma$ induced by $A$
is the same as the complement of the subgraph of $\Gamma$ induced by $A$.
\item[(ii)] If $\Gamma_0$ is an induced subgraph of $\Gamma$
and $B\subset V(\Gamma_0)$, then $\Gamma[B]=\Gamma_0[B]$,
i.e.\ $B$ induces the same subgraph in $\Gamma$ and in $\Gamma_0$.
\end{itemize}

For graphs $\Gamma_1$ and $\Gamma_2$,
the \emph{disjoint union} $\Gamma_1\sqcup \Gamma_2$
is the graph such that
\begin{align*}
V(\Gamma_1\sqcup\Gamma_2) &=V(\Gamma_1)\sqcup V(\Gamma_2),\\
E(\Gamma_1\sqcup\Gamma_2) &=E(\Gamma_1)\sqcup E(\Gamma_2).
\end{align*}

The \emph{join} $\Gamma_1 * \Gamma_2$ is the graph such that
$\overline{\Gamma_1 * \Gamma_2} =  \overline{\Gamma_1} \sqcup \overline{\Gamma_2}$, hence
\begin{align*}
V(\Gamma_1*\Gamma_2) &=V(\Gamma_1)\sqcup V(\Gamma_2),\\
E(\Gamma_1*\Gamma_2) &=E(\Gamma_1)\sqcup E(\Gamma_2)\sqcup
\{\,\{v_1,v_2\}: v_1\in V(\Gamma_1),~v_2\in V(\Gamma_2)\,\}.
\end{align*}

A graph is called a {\em join} if it is the join of two nonempty graphs.
A subgraph that is a join is called a {\em subjoin}.

The \emph{path graph} $P_n=P_n(v_1,\ldots,v_n)$
is the graph with $V(P_n)=\{v_1,\ldots,v_n\}$
and $E(P_n)=\{ \{v_i,v_{i+1}\}:1 \leqslant i \leqslant n-1\}$.
For example, $P_4(v_1,v_2,v_3,v_4)$ looks like
$\begin{xy}/r.7mm/:
(0,0)="a" *+!U{v_1} *{\bullet},
"a"+(10,0)="a" *+!U{v_2} *{\bullet},
"a"+(10,0)="a" *+!U{v_3} *{\bullet},
"a"+(10,0)="a" *+!U{v_4} *{\bullet},
(0,0); "a" **@{-},
\end{xy}$
and its complement is $P_4(v_3,v_1,v_4,v_2)$.
Notice that $P_n$ is biconnected for all $n\ge 4$ and
that $P_4$ is the only biconnected graph with four vertices.

\subsection{Right-angled Artin groups}

An element of $V(\Gamma)^{\pm 1}=V(\Gamma)\cup V(\Gamma)^{-1}$
is called a \emph{letter}.
A \emph{word} means a finite sequence of letters.
Suppose that $g\in A(\Gamma)$ is expressed as a word $w$.
The word $w$ is called \emph{reduced} if $w$ is a shortest
word among all the words representing $g$.
In this case, the length of $w$ is called the \emph{word length} of $g$.

\begin{definition}[support]
For $g\in A(\Gamma)$, the \emph{support} of $g$, denoted $\supp(g)$,
is the set of generators that appear in a reduced word representing $g$.
\end{definition}
It is known that $\supp(g)$ is well defined (see~\cite{HM95}),
i.e.\ it does not depend on the choice of a reduced word
representing $g$.

\begin{definition}[cyclically reduced]
An element $g\in A(\Gamma)$ is called \emph{cyclically reduced}
if it has the minimal word length in its conjugacy class.
\end{definition}

Notice that every positive word
(i.e.\ a word composed of only letters in $V(\Gamma)$)
is both reduced and cyclically reduced.

The following equivalences are well known (see \cite{KK14b, LL22} for example).

\begin{lemma}\label{lem:loxo}
Let $\Gamma$ be a nontrivial biconnected graph.
For a cyclically reduced element $g\in A(\Gamma)\setminus\{1\}$,
the following are equivalent:
\begin{itemize}
\item[(i)] $g$ is loxodromic on $(\Gamma^e, d_e)$;
\item[(ii)] $\supp(g)$ is not contained in a subjoin of $\Gamma$;
\item[(iii)]
$\bar\Gamma[\supp(g)]$ is connected
and $\supp(g)$ dominates $\bar\Gamma$.
\end{itemize}
\end{lemma}

\section{Biconnected graphs}

In this section we study biconnected graphs.
Theorem~\ref{thm:2v} is the main result of this section.
We begin with three easy lemmas.

A vertex $v$ of $\Gamma$ is called
a \emph{cut vertex} of $\Gamma$ if $\Gamma\setminus v$ is disconnected.

\begin{lemma}\label{lem:cut}
Let $\Gamma$ be a biconnected graph.
If\/ $\Gamma$ has a cut vertex, then $\diam(\Gamma)\ge 3$.
\end{lemma}

\begin{proof}
Let $v_0$ be a cut vertex of $\Gamma$.
Then $\Gamma\setminus v_0$ is a disjoint union $\Gamma_1\sqcup \Gamma_1$
for some nontrivial subgraphs $\Gamma_1$ and $\Gamma_2$.
(See Figure~\ref{fig:cut}.)
If $d_{\Gamma}(v_0,v)\le 1$ for all $v\in V(\Gamma)$,
then $\Gamma=\{v_0\}*(\Gamma\setminus v_0)$,
which contradicts that $\bar\Gamma$ is connected.
Therefore there exists a vertex $v_1$ such that $d_{\Gamma}(v_0,v_1)\ge 2$.
We may assume $v_1\in\Gamma_1$.
Choose a vertex $v_2$ of $\Gamma_2$.
Because $v_0$ is a cut vertex, each path from $v_1$ to $v_2$ must pass through $v_0$.
Therefore $d_{\Gamma}(v_1,v_2)=d_{\Gamma}(v_1,v_0)+d_{\Gamma}(v_0,v_2)\ge 2+1= 3$,
hence $\diam(\Gamma)\ge 3$.
\end{proof}

\begin{figure}
$$
\begin{xy}/r1mm/:
(0,0)="or" *+!U{v_0} *{\bullet}, (15,0)="h", (0,3)="v",
"or";
"or"+"h"="po" **@{-}  *{\bullet},
"po"+(8,0) *\xycircle(13,8){-},
"po"+(21,-5) *!L{\Gamma_2},
"po"+"v" **@{-} *{\bullet},
"po"-"v"="v2" **@{-}  *+!L{v_2} *{\bullet},
"or"-"h"="po" **@{-} *{\bullet},
"po"-(8,0) *\xycircle(13,8){-},
"po"+(-20,-5) *!R{\Gamma_1},
"po"-"v" **@{-} *{\bullet},
"po"+"v"="v1" **@{-} *{\bullet};
"v1"-"h"="v1" **@{-}  *+!U{v_1} *{\bullet},
\end{xy}
$$
\caption{Figure for Lemma~\ref{lem:cut}.}\label{fig:cut}
\end{figure}
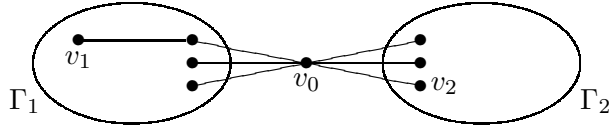

\begin{lemma}\label{lem:conn}
Let $A$ and $B$ be subsets of\/ $V(\Gamma)$
such that $A\subset B\subset\St_\Gamma(A)$.
If\/ $\Gamma[A]$ is connected then $\Gamma[B]$ is also connected.
\end{lemma}

\begin{proof}
Since $A\subset B$, the graph $\Gamma[B]$ is obtained from the connected graph $\Gamma[A]$
by adding some vertices and some edges.
Since $B\subset\St_\Gamma(A)$, each vertex in $\Gamma[B]\setminus\Gamma[A]$
is adjacent to a vertex in $\Gamma[A]$,
hence $\Gamma [B]$ is connected.
\end{proof}

\begin{lemma}\label{lem:domi}
For any $v\in V(\Gamma)$,
$\St_\Gamma(v)$ dominates $\bar\Gamma$.
For any $e\in E(\Gamma)$,
$\St_\Gamma(e)$ dominates $\bar\Gamma$.
\end{lemma}

\begin{proof}
For a vertex $v$,
each $u\in V(\Gamma)\setminus\{v\}$ is adjacent to $v$ either in $\Gamma$ or
in $\bar\Gamma$, hence
$V(\Gamma)=\St_\Gamma(v)\cup \St_{\bar\Gamma}(v)
\subset \St_{\bar\Gamma}(\St_\Gamma(v))$.
This means that $\St_\Gamma(v)$ dominates $\bar\Gamma$.
For an edge $e=\{v_1,v_2\}$,
$\St_\Gamma(e)$ dominates $\bar\Gamma$
because $\St_\Gamma(v_1)\subset\St_\Gamma(e)$
and $\St_\Gamma(v_1)$ dominates $\bar\Gamma$.
\end{proof}

\begin{lemma}\label{lem:diam3}
Let $\Gamma$ be a biconnected graph.
If either\/ $\diam(\Gamma)\ge 3$ or $\diam(\bar\Gamma)\ge 3$,
then there exists an edge $e$ in $\Gamma$ such that
$\Gamma[\St_{\Gamma}(e)]$ is biconnected.
\end{lemma}

\begin{proof}
Since $\Gamma[\St_\Gamma(e)]$ is connected for any edge $e$,
it suffices to show that $\bar\Gamma[\St_\Gamma(e)]$ is connected
for some edge $e$ of $\Gamma$.
We use the following claim.

\medskip\noindent
\textit{Claim.}\ \
Let $v_1,v_2\in V(\Gamma)$.
\begin{itemize}
\item[(i)]
If\/ $d_{\bar\Gamma}(v_1,v_2)\ge 3$,
then $\{v_1,v_2\}$ dominates $\Gamma$,
i.e.\  $\St_\Gamma(\{v_1,v_2\})=V(\Gamma)$.
\item[(ii)]
If\/ $d_\Gamma(v_1,v_2)\ge 3$,
then $\{v_1,v_2\}$ dominates $\bar\Gamma$,
i.e.\ $\St_{\bar\Gamma}(\{v_1,v_2\})=V(\Gamma)$.
\end{itemize}

\begin{proof}[Proof of Claim]
Since (i) and (ii) are equivalent, we prove only (i).
Suppose $d_{\bar\Gamma}(v_1,v_2)\ge 3$.
For each $v\in V(\Gamma)$, either $d_{\bar\Gamma}(v,v_1)\ge 2$
or $d_{\bar\Gamma}(v,v_2)\ge 2$ because otherwise
$d_{\bar\Gamma}(v_1,v_2)
\le d_{\bar\Gamma}(v_1,v)+ d_{\bar\Gamma}(v,v_2)\le 1+1=2$.
Therefore either $\{v,v_1\}\in E(\Gamma)$ or
$\{v,v_2\}\in E(\Gamma)$.
This means that $\St_\Gamma(\{v_1,v_2\})=V(\Gamma)$.
\end{proof}

We first assume $\diam(\bar\Gamma)\ge 3$.
Choose $v_1,v_2\in V(\Gamma)$
such that $d_{\bar\Gamma}(v_1,v_2)\ge 3$,
and let  $e=\{v_1,v_2\}$.
By the above claim, $\St_\Gamma(e)=V(\Gamma)$,
hence $\bar\Gamma[\St_\Gamma(e)]= \bar\Gamma$ and it is connected.

We now assume $\diam(\Gamma)\ge 3$.
Choose $v_1,v_4\in V(\Gamma)$ such that $d_\Gamma(v_1,v_4)=3$.
Choose a path $(v_1,v_2,v_3,v_4)$ from $v_1$ to $v_4$,
and let $e=\{v_2,v_3\}$ and $A=\{v_1,v_2,v_3,v_4\}$.
Then $\Gamma[A]$ is the path graph $P_4(v_1,v_2,v_3,v_4)$,
hence $\bar\Gamma[A]$ is the path graph $P_4(v_2,v_4,v_1,v_3)$,
which is connected.

Notice that $\{ v_1, v_4\}\subset A\subset\St_\Gamma(e)$.
Since $d_\Gamma(v_1,v_4)\ge 3$, $\{v_1,v_4\}$ dominates $\bar\Gamma$
by the above claim,
hence $A$ also dominates $\bar\Gamma$.
Therefore
$$
A\subset\St_\Gamma(e)\subset  V(\Gamma) =\St_{\bar\Gamma}(A).
$$
Since $\bar\Gamma[A]$ is connected,
$\bar\Gamma[\St_\Gamma(e)]$ is also connected
(by Lemma~\ref{lem:conn}).
\end{proof}

Using the above lemmas, we establish the following theorem.

\begin{theorem}\label{thm:2v}
For any nontrivial biconnected graph $\Gamma$,
there exists an edge $e$ of\/ $\Gamma$ such that
$\Gamma[\St_\Gamma(e)]$ is biconnected.
\end{theorem}

Before proving the theorem, let us see examples.

\begin{example}\label{ex:ex3}
Let $\Gamma$ be the path graph $P_6$ in Figure~\ref{fig:ex3}(a).
If $e=\{v_1,v_2\}$, then $\Gamma[\St_{\Gamma}(e)]$ is the path graph
$P_3(v_1,v_2,v_3)$, which is not biconnected.
If $e=\{v_2,v_3\}$, then $\Gamma[\St_{\Gamma}(e)]$ is the path graph
$P_4(v_1,v_2,v_3,v_4)$, which is biconnected.
It is easy to see that $\Gamma[\St_{\Gamma}(e)]$ is biconnected
if and only if $e=\{v_i,v_{i+1}\}$ for $2\le i\le 4$.

Let $\Gamma$ be the graph with 7 vertices in Figure~\ref{fig:ex3}(b).
If $e=\{v_3,v_4\}$,
then $\Gamma[\St_{\Gamma}(e)]$
is the join $\{v_0\}*P_4(v_2,v_3,v_4,v_5)$,
which is not biconnected.
If $e=\{v_0,v_6\}$,
then $\Gamma[\St_{\Gamma}(e)]$
is the join $\{v_0\}*P_5(v_2,v_3,v_4,v_5,v_6)$,
which is not biconnected.
It is easy to see that
$\Gamma[\St_{\Gamma}(e)]$ is biconnected
if and only if $e$ is $\{v_0,v_2\}$ or $\{v_2,v_3\}$.

Let $\Gamma$ be the pentagon $C_5$ in Figure~\ref{fig:ex3}(c).
Then $\Gamma[\St_\Gamma(e)]$ is biconnected for any edge $e$.
\end{example}

\begin{figure}
\begin{tabular}{ccccc}
\begin{xy}/r.8mm/:
(0,5)="or"="a"="a1" *+!U{v_1} *{\bullet},
"a"+(10,0)="a"="a2" *+!U{v_2} *{\bullet},
"a"+(10,0)="a"="a3" *+!U{v_3} *{\bullet},
"a"+(10,0)="a"="a4" *+!U{v_4} *{\bullet},
"a"+(10,0)="a"="a5" *+!U{v_5} *{\bullet},
"a"+(10,0)="a"="a6" *+!U{v_6} *{\bullet},
"a1"; "a" **@{-},
\end{xy}
&& 
\begin{xy}/r.8mm/:
(0,5)="or"="a"="a1" *+!U{v_1} *{\bullet},
"a"+(10,0)="a"="a2" *+!U{v_2} *{\bullet},
"a"+(10,0)="a"="a3" *+!U{v_3} *{\bullet},
"a"+(10,0)="a"="a4" *+!U{v_4} *{\bullet},
"a"+(10,0)="a"="a5" *+!U{v_5} *{\bullet},
"a"+(10,0)="a"="a6" *+!U{v_6} *{\bullet},
"a1"; "a" **@{-},
"or"+(25, 27)="v0" *+!D{v_0} *{\bullet}, "v0";
"a2" **@{-},
"a3" **@{-},
"a4" **@{-},
"a5" **@{-},
"a6" **@{-},
\end{xy} &&
\begin{xy}/r1.2mm/:
(0,13)="or" ,
"or"+( 0, 10)="v1" *{\bullet} *+!D{v_1},
"or"+(-9.51,  3.09)="v2" *{\bullet} *+!R{v_2},
"or"+(-5.88, -8.09)="v3" *{\bullet} *+!U{v_3},
"or"+( 5.88, -8.09)="v4" *{\bullet} *+!U{v_4},
"or"+( 9.51,  3.09)="v5" *{\bullet} *+!L{v_5},
"v1";
"v2" **@{-};
"v3" **@{-};
"v4" **@{-};
"v5" **@{-};
"v1" **@{-};
\end{xy}\\
(a) Path graph $P_6$  &&
(b) A graph with 7 vertices &&
(c) Pentagon $C_5$
\end{tabular}
\caption{Graphs for Example~\ref{ex:ex3}}\label{fig:ex3}
\end{figure}
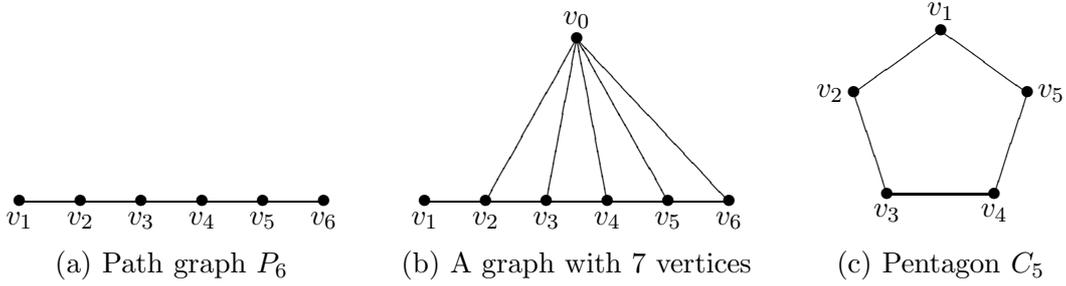

Figures~\ref{fig:c1} and \ref{fig:c21} are
for the proof of Theorem~\ref{thm:2v},
where the solid (resp.\ dotted) lines between two vertices
indicate edges (resp.\ non-edges) in $\Gamma$.

\begin{proof}[Proof of Theorem~\ref{thm:2v}]
Since $\Gamma[\St_\Gamma(e)]$ is connected for any edge $e$,
it suffices to show that $\bar\Gamma[\St_\Gamma(e)]$ is connected
for some edge $e$.

We use an induction on $|V(\Gamma)|$.

If $|V(\Gamma)|=4$, then $\Gamma$ must be the path graph $P_4$,
hence we are done (by Lemma~\ref{lem:diam3}).

We now assume $|V(\Gamma)|\ge 5$.
Choose a vertex $v_0$, and let $\Gamma_0=\Gamma\setminus v_0$.

By Lemma~\ref{lem:cut},
if $\Gamma_0$ is disconnected then $\diam(\Gamma)\ge 3$,
and if $\overline{\Gamma_0}$ is disconnected then $\diam(\bar\Gamma)\ge 3$.
In either case, we are done (by Lemma~\ref{lem:diam3}).
Therefore we may assume
that $\Gamma_0$ is biconnected.

By induction hypothesis, there exists an edge $e_0=\{v_1,v_2\}$
of $\Gamma_0$ such that
$\Gamma_0[\St_{\Gamma_0}(e_0)]$ is biconnected.
Notice the following:
\begin{itemize}
\item $\St_\Gamma(e_0)$ is either $\St_{\Gamma_0}(e_0)$ or
$\St_{\Gamma_0}(e_0)\cup\{v_0\}$;

\item $\bar\Gamma[\St_{\Gamma_0}(e_0)]
=\overline{\Gamma_0}[\St_{\Gamma_0}(e_0)]
=\overline{\Gamma_0[\St_{\Gamma_0}(e_0)]}$;

\item $\overline{\Gamma_0}[\St_{\Gamma_0}(e_0)]$
is connected because $\Gamma_0[\St_{\Gamma_0}(e_0)]$ is biconnected.
\end{itemize}

\par\medskip 
\textit{Case 1.} $v_0$ does not dominate $\St_{\Gamma_0}(e_0)$ in $\Gamma$.
\par\smallskip

We will show that in this case $\bar\Gamma[\St_\Gamma(e_0)]$ is connected.
Then we are done by taking $e=e_0$.

If $v_0$ is adjacent to neither $v_1$ nor $v_2$ in $\Gamma$,
then $\St_\Gamma(e_0)=\St_{\Gamma_0}(e_0)$,
hence
$$
\bar\Gamma[\St_\Gamma(e_0)]
=\bar\Gamma[\St_{\Gamma_0}(e_0)]
=\overline{\Gamma_0}[\St_{\Gamma_0}(e_0)].$$
Since $\overline{\Gamma_0}[\St_{\Gamma_0}(e_0)]$ is connected,
we are done.

If $v_0$ is adjacent to either $v_1$ or $v_2$ in $\Gamma$,
then
$\St_\Gamma(e_0)=\St_{\Gamma_0}(e_0)\cup \{v_0\}$.
Since $v_0$ does not dominate $\St_{\Gamma_0}(e_0)$ in $\Gamma$,
there is a vertex $v_3\in\St_{\Gamma_0}(e_0)$
that is not adjacent to $v_0$ in $\Gamma$
(hence $v_0$ is adjacent to $v_3$ in $\bar\Gamma$).
See Figure~\ref{fig:c1}.
Since $\bar\Gamma[\St_{\Gamma}(e_0)]$ is obtained from
$\overline{\Gamma_0}[\St_{\Gamma_0}(e_0)]$ by adding a vertex $v_0$
and adding at least one edge $\{v_0,v_3\}$,
the graph $\bar\Gamma[\St_{\Gamma}(e_0)]$ is connected,
hence we are done.

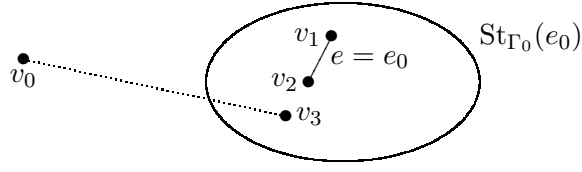
\begin{figure}
$$
\begin{xy}/r1.5mm/:
(0,0)="or" *{\bullet}, *+!U{v_0},
"or";
"or"+(27, 2)="v1" *{\bullet} *+!R{v_1},
"or"+(25,-2)="v2" *{\bullet} *+!R{v_2},
"or"+(23,-5)="v3" **@{.} *{\bullet} *+!L{v_3},
"or"+(28, -2) *\xycircle(12,7){-},
"v1"; "v2" **@{-} ?(.5) *+!L{e=e_0},
"or"+(40,2) *!L{\St_{\Gamma_0}(e_0)},
\end{xy}
$$
\caption{Figure for Case 1 in the proof of Theorem~\ref{thm:2v}}
\label{fig:c1}
\end{figure}

\par\medskip 
\textit{Case 2.} $v_0$ dominates $\St_{\Gamma_0}(e_0)$ in $\Gamma$.
\par\smallskip

Notice that $\St_{\Gamma}(e_0) = \St_{\Gamma_0}(e_0)\cup\{v_0\}
\subset \St_\Gamma(v_0)$.

There exists a vertex $v_3 \in V(\Gamma)\backslash \St_{\Gamma_0}(e_0)$ such that $d_\Gamma(v_0,v_3)=2$
because otherwise $d_\Gamma(v_0,v)\le 1$ for all $v\in V(\Gamma)$,
hence $\Gamma=\{v_0\}*\Gamma_0$,
which contradicts that $\bar\Gamma$ is connected.
Choose $v_4\in V(\Gamma)$ such that
$(v_0,v_4,v_3)$ is a path in $\Gamma$ from $v_0$ to $v_3$.
See Figure~\ref{fig:c21},
where $v_4\in\St_{\Gamma_0}(e_0)$ in the left
and $v_4\not\in\St_{\Gamma_0}(e_0)$ in the right.

Let $e=\{v_0,v_4\}$.
We will show that $\bar\Gamma[\St_\Gamma(e)]$ is connected.

Let $A=\St_{\Gamma_0}(e_0)\cup\{v_0,v_3,v_4\}$,
and observe the following.
\begin{itemize}
\item
$A$ dominates $\bar\Gamma$ because
$\St_{\Gamma}(e_0)\subset A$
and
$\St_{\Gamma}(e_0)$ dominates $\bar\Gamma$ (by Lemma~\ref{lem:domi}).
Therefore
$$A\subset \St_\Gamma(e)\subset V(\Gamma) =\St_{\bar\Gamma}(A).$$

\item
Since $\overline{\Gamma_0}[\St_{\Gamma_0}(e_0)]$ is connected
and since $\{v_1,v_3\}$ and $\{v_3,v_0\}$ are edges of $\bar\Gamma$,
the graph $\bar\Gamma[\St_{\Gamma_0}(e_0)\cup\{v_0,v_3\}]$ is connected.
Since either $v_4$ is contained in $\St_{\Gamma_0}(e_0)$
(as in Figure~\ref{fig:c21}(a))
or $\{v_1,v_4\}$ is an edge of $\bar\Gamma$ (as in Figure~\ref{fig:c21}(b)),
the graph $\bar\Gamma[\St_{\Gamma_0}(e_0)\cup\{v_0,v_3,v_4\}]
=\bar\Gamma[A]$ is connected.
\end{itemize}
Since $A\subset \St_\Gamma(e)\subset\St_{\bar\Gamma}(A)$
and $\bar\Gamma[A]$ is connected,
$\bar\Gamma[\St_\Gamma(e)]$ is connected
(by Lemma~\ref{lem:conn}).
\end{proof}

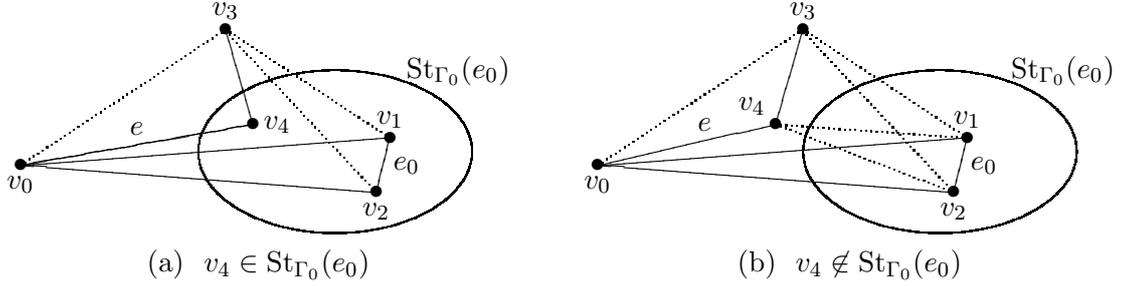
\begin{figure}
$$\begin{array}{ccc}
\begin{xy}/r1.8mm/:
(0,0)="or" *{\bullet}, *+!U{v_0},
"or";
"or"+(17, 3)="v4" **@{-} *{\bullet} *+!L{v_4},
"or"+(26,-2)="v2" **@{-} *{\bullet} *+!U{v_2},
"or"+(27, 2)="v1" **@{-} *{\bullet} *+!D{v_1},
"or"+(15,10)="v3" *{\bullet} *+!D{v_3},
"v1"; "v2" **@{-} ?(.5) *+!L{e_0},
"v4"; "v3" **@{-}, "or" **@{-} ?(.5) *+!D{e},
"v3"; "v1" **@{.}, "v2" **@{.}, "or" **@{.},
"or"+(23,1) *\xycircle(10,6){-},
"or"+(32,7) *+{\St_{\Gamma_0}(e_0)},
\end{xy}
& &
\begin{xy}/r1.8mm/:
(0,0)="or" *{\bullet}, *+!U{v_0},
"or";
"or"+(27, 2)="v1" **@{-} *{\bullet} *+!D{v_1},
"or"+(26,-2)="v2" **@{-} *{\bullet} *+!U{v_2},
"or"+(13,3)="v4" *{\bullet} *+!DR{v_4},
"or"+(15,10)="v3" *{\bullet} *+!D{v_3},
"v1"; "v2" **@{-} ?(.5) *+!L{e_0},
"v4"; "v1" **@{.}, "v2" **@{.},"v3" **@{-}, "or" **@{-} ?(.4) *+!D{e},
"v3"; "v1" **@{.}, "v2" **@{.}, "or" **@{.},
"or"+(25,1) *\xycircle(10,6){-},
"or"+(34,7) *+{\St_{\Gamma_0}(e_0)},
\end{xy}\\[2em]
\mbox{(a)\ \ $v_4\in \St_{\Gamma_0}(e_0)$}
&&
\mbox{(b)\ \ $v_4\not\in \St_{\Gamma_0}(e_0)$}
\end{array}
$$
\caption{Figures for Case 2 in the proof of Theorem~\ref{thm:2v}}
\label{fig:c21}
\end{figure}

The following proposition will be used in the proof of Theorem~\ref{thm:V_TL}.

\begin{proposition}\label{prop:22}
For each $n\ge 3$, there exists a biconnected graph $\Lambda_n$
satisfying the following:
\begin{itemize}
\item[(i)] $\diam(\Lambda_n)=\diam(\overline{\Lambda_n})=2$ and $|V(\Lambda_n)|=n+2$;
\item[(ii)]  $\Lambda_n$ contains a path graph $P_n$ as an induced subgraph
in a way that $V(P_n)$ dominates both $\Lambda_n$ and $\overline{\Lambda_n}$.
\end{itemize}
\end{proposition}

\begin{proof}
If $n=3$, let $\Lambda_3$ be the pentagon
such that $P_3(v_1,v_2,v_3)$ is an induced subgraph of $\Lambda_3$
as in the left of Figure~\ref{fig:22}.
It is straightforward to see that
$\Lambda_3$ has the desired properties.

If $n\ge 4$, let $\Lambda_n$ be the graph obtained from $P_n=P_n(v_1,\ldots,v_n)$
by adding two vertices $x$ and $y$
and by adding edges $\{x,v_i\}$ for $i\in\{2,\ldots,n\}$
and $\{y,v_j\}$ for $j \in\{1,4,5,\ldots,n\}$
as in the right of Figure~\ref{fig:22} for the case of $n=7$.
It is straightforward to see that $\Lambda_n$ has the desired properties.
\end{proof}

\begin{figure}
$$
\begin{xy}/r.8mm/:
(0,-10)="or"="a"="a1" *+!U{v_1} *{\bullet},
"a"+(15,0)="a"="a2" *+!U{v_2} *{\bullet},
"a"+(15,0)="a"="a3" *+!U{v_3} *{\bullet},
"a1"; "a" **@{-},
"or"+(0,15)="x" *+!D{x} *{\bullet}, "x"; "a1" **@{-},
"a3"+(0,15)="y" *+!D{y} *{\bullet}, "y"; "a3" **@{-}, "x" **@{-},
\end{xy}
\qquad\qquad
\begin{xy}/r.8mm/:
(0,0)="or"="a"="a1" *+!U{v_1} *{\bullet},
"a"+(10,0)="a"="a2" *+!U{v_2} *{\bullet},
"a"+(10,0)="a"="a3" *+!U{v_3} *{\bullet},
"a"+(10,0)="a"="a4" *+!U{v_4} *{\bullet},
"a"+(10,0)="a"="a5" *+!U{v_5} *{\bullet},
"a"+(10,0)="a"="a6" *+!U{v_6} *{\bullet},
"a"+(10,0)="a"="a7" *+!U{v_7} *{\bullet},
"a1"; "a" **@{-},
"or"+(15, 15)="x" *+!D{x} *{\bullet}, "x";
"a2" **@{-},
"a3" **@{-},
"a4" **@{-},
"a5" **@{-},
"a6" **@{-},
"a7" **@{-},
"or"+(15,-20)="y" *+!U{y} *{\bullet}, "y";
"a1" **@{-},
"a4" **@{-},
"a5" **@{-},
"a6" **@{-},
"a7" **@{-},
\end{xy}
$$
\caption{Figures for Proposition~\ref{prop:22}}\label{fig:22}
\end{figure}
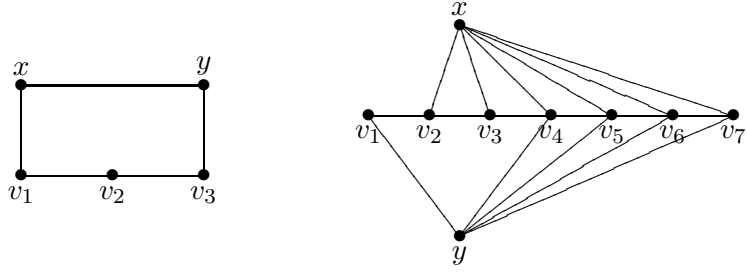

\section{Minimal asymptotic translation lengths}

In this section, we establish our results on
the minimal asymptotic translation lengths
of right-angled Artin groups on extension graphs.
Let us abbreviate the translation length
$\tau_{(\Gamma^e,d_e)}(g)$ to $\tau_e(g)$.

For $v\in V(\Gamma)$, let $Z(v)$ denote the \emph{centralizer} of $v$
in $A(\Gamma)$, i.e.\
$Z(v)=\{g\in A(\Gamma): gv=vg~\mbox{in $A(\Gamma)$}\}$.
It is well known that $Z(v)$ is generated by $\St_\Gamma(v)$.

\begin{lemma}\label{lem:tl2}
Let $\{v_1,v_2\} \in E(\Gamma)$ and $g_1, g_2\in A(\Gamma)$ such that
$g_1\in Z(v_1)$ and $g_2\in Z(v_2)$, and let $g=g_1g_2$.
Then $\tau_e(g)\le 2$.
\end{lemma}

\begin{proof}
Notice that $\{v_1^h,v_2^h\}\in E(\Gamma^e)$ for
any $h\in A(\Gamma)$.
In particular, $\{v_1^{g_2},v_2^{g_2}\}\in E(\Gamma^e)$.
Since $v_1^{g_2}=(v_1^{g_1})^{g_2}=v_1^{g_1g_2}=v_1^g$
and $v_2^{g_2}=v_2$,
we have $\{v_1^g,v_2\}\in E(\Gamma^e)$.

Now $(v_1,v_2,v_1^g)$ is a path in $\Gamma^e$, hence
$(v_1,v_2,v_1^g,v_2^g,v_1^{g^2},v_2^{g^2},v_1^{g^3},\ldots)$
is also a path in $\Gamma^e$. See Figure ~\ref{fig:len2}.
Therefore $d_e(v_1,v_1^{g^n})\le2n$ for all $n\ge1$,
which implies $\tau_e(g)\le 2$.
\end{proof}

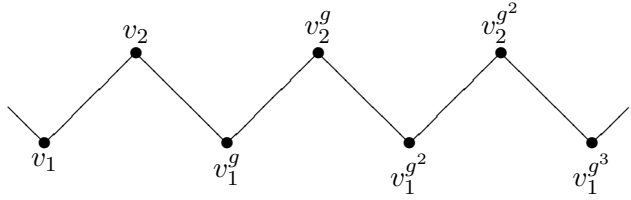
\begin{figure}
$$
\begin{xy}/r1.2mm/:
(10, 10)="hd1",
(10,-10)="hd2",
(0,0)="or",
     "or"="v1" *{\bullet} *+!U{v_1},
"v1"+"hd1"="v2" *{\bullet} *+!D{v_2};   "v1" **@{-},
"v2"+"hd2"="v1" *{\bullet} *+!U{v_1^g}; "v2" **@{-},
"v1"+"hd1"="v2" *{\bullet} *+!D{v_2^g}; "v1" **@{-},
"v2"+"hd2"="v1" *{\bullet} *+!U{v_1^{g^2}}; "v2" **@{-},
"v1"+"hd1"="v2" *{\bullet} *+!D{v_2^{g^2}}; "v1" **@{-},
"v2"+"hd2"="v1" *{\bullet} *+!U{v_1^{g^3}}; "v2" **@{-},
"v1"; "v1"+( 4,4) **@{-},
"or"; "or"+(-4,4) **@{-},
\end{xy}
$$
\caption{Figure for Lemma~\ref{lem:tl2}}\label{fig:len2}
\end{figure}

In the above lemma, the element $g$ may not be loxodromic.
The following lemma gives a sufficient condition
under which a loxodromic element $g$ with $\tau_e(g)\le 2$ exists.

\begin{lemma}\label{lem:len2}
Let $\Gamma$ be a nontrivial biconnected graph.
If there exists an edge $e$ of\/ $\Gamma$
such that $\Gamma[\St_\Gamma(e)]$ is biconnected,
then there exists a loxodromic element $g\in A(\Gamma)$
such that $\tau_e(g)\le 2$.
\end{lemma}

\begin{proof}
Let $e=\{v_1,v_2\}$,
$\St_{\Gamma}(v_1)=\{v_1,x_1,\ldots,x_p\}$
and $\St_{\Gamma}(v_2)=\{v_2,y_1,\ldots,y_q\}$.
Let
$$
g_1=v_1x_1\cdots x_p,\quad
g_2=v_2y_1\cdots y_q,\quad
g=g_1g_2.
$$
Then $g_i\in Z(v_i)$ for $i=1,2$,
hence $\tau_e(g)\le 2$ (by Lemma~\ref{lem:tl2}).
Therefore it suffices to show that $g$ is loxodromic.

Since each $g_i$ is represented
by a positive word on generators,
$g$ is cyclically reduced.
Notice that
$$
\supp(g)=\supp(g_1)\cup\supp(g_2)
=\St_\Gamma(v_1)\cup \St_\Gamma(v_2)
=\St_\Gamma(e).
$$
Since $\St_\Gamma(e)$ dominates $\bar\Gamma$
(by Lemma~\ref{lem:domi})
and $\bar\Gamma[\St_\Gamma(e)]$ is connected
by the hypothesis,
$g$ is loxodromic on $\Gamma^e$ (by Lemma~\ref{lem:loxo}).
\end{proof}

Combining the above lemma with Theorem~\ref{thm:2v} yields the following theorem.

\begin{theorem}\label{thm:tl}
For any nontrivial biconnected graph $\Gamma$,
$$\minTL_{(\Gamma^e, d_e)}(A(\Gamma))\le 2.$$
\end{theorem}

\begin{proof}
By Theorem~\ref{thm:2v},
there exists an edge $e=\{v_1,v_2\}$ in $\Gamma$
such that $\Gamma[\St_\Gamma(e)]$ is biconnected.
By Lemma~\ref{lem:len2}, there exists a loxodromic element $g\in A(\Gamma)$
such that $\tau_e(g)\le 2$.
\end{proof}

In Theorem~\ref{thm:d_TL}, we will obtain a better upper bound
for the minimal asymptotic translation length for the case where
$\diam(\bar\Gamma)$ is large.
For this, we first show the following lemma.

\begin{lemma}\label{lem:minlox}
Let $\Gamma$ be a nontrivial biconnected graph.
Suppose that there exists an induced subgraph
$\Lambda$ of the complement graph $\bar\Gamma$ such that
\begin{itemize}
\item[(i)] $\Lambda$ is connected and $d=\diam(\Lambda)\ge 3$;
\item[(ii)] $V(\Lambda)$ dominates $\bar\Gamma$.
\end{itemize}
Then there exists a loxodromic element $g\in A(\Gamma)$
such that $\tau_e(g)\le \frac{2}{d -2}$.
Therefore
$$
\minTL_{(\Gamma^e,d_e)} (A(\Gamma)) \le \frac{2}{d -2}.
$$
\end{lemma}

\begin{myproof}
Since $d =\diam(\Lambda) \ge 3$, there are vertices
$v_0,v_d \in V(\Lambda)$ such that
$d_{\Lambda}(v_0,v_d )=d \ge 3$.

For each $0\le k\le d $, let $V_k=\{v\in V(\Lambda): d_{\Lambda}(v_0,v)=k\}$
(see Figure~\ref{fig:minlox}).
Then each $V_k$ is nonempty and $V(\Lambda)=V_0 \sqcup \cdots\sqcup V_d$.
Let $V_k=\{v_{k1},v_{k2},\ldots v_{kr_k}\}$.
Define
\begin{align*}
g_k&=v_{k1}\cdots v_{kr_k}\quad \mbox{for $0\le k\le d$},\\
g&=g_0g_1\cdots g_d ,\\
h_0&=(g_0g_1\cdots g_{d -2})(g_0g_1\cdots g_{d -3}) \cdots (g_0g_1 g_2) (g_0g_1),\\
h_d &=(g_{d -1} g_{d }) (g_{d -2} g_{d -1} g_{d }) \cdots
(g_3\cdots g_d )(g_2\cdots g_d ).
\end{align*}
Observe the following.
\begin{itemize}
\item $\supp(g_k)=V_k$ for each $0\le k\le d$, hence $g_ig_j=g_jg_i$ if $|i-j|\ge 2$;
\item $g$ is cyclically reduced and $\supp(g)=V(\Lambda)$;
\item since $\bar{\Gamma}[\supp(g)]=\Lambda$ is connected
and $\supp(g)=V(\Lambda)$ dominates $\bar\Gamma$,
$g$ is loxodromic on $(\Gamma^e, d_e)$ (by Lemma~\ref{lem:loxo});
\item $g_0,\ldots,g_{d -2}\in Z(v_d )$, hence $h_0\in Z(v_d )$;
\item $g_2,g_3,\ldots, g_d \in Z(v_0)$, hence $h_d \in Z(v_0)$.
\end{itemize}
A straightforward computation using the commutativity
$g_ig_j=g_jg_i$ for $|i-j|\ge 2$ shows that
$$
h_0h_d=g^{d -2}.
$$
Since $d_{\Lambda}(v_0, v_d) \ge 3$,
$\{v_0,v_d\}$ is an edge of $\bar\Lambda$ and hence of $\Gamma$.
Since $h_0\in Z(v_d )$, $h_d \in Z(v_0)$ and $g^{d-2}=h_0h_d$,
we have $\tau_e(g^{d-2})\le 2$ (by Lemma~\ref{lem:tl2}).
Therefore $\tau_e(g) \le \frac{2}{d -2}$.
Since $g$ is loxodromic, we are done.
\end{myproof}

\begin{figure}
$$
\begin{xy}/r1mm/:
( 0,0)="or" *+!U{v_0} *{\bullet},
(20,0)="hd",
"or" *\xycircle(5,12){-}, "or"+(0,-10) *++!UL{V_0},
"or"; "or"+"hd"="h",
"h"+(0,6)="a1" **@{-} *{\bullet} ,
"h"+(0,3)="a2" **@{-} *{\bullet} ,
"h"+(0,0) *{\vdots},
"h"+(0,-5)="a3" **@{-} *{\bullet},
"h"+(0,-8)="a4" **@{-} *{\bullet},
"h"+(0,0) *\xycircle(5,12){-}, "h"+(0,-10) *++!UL{V_1},
"h"; "h"+"hd"="h",
"h"+(0,6)="a1" *{\bullet}, "a1"; "a1"-"hd", **@{-},
"h"+(0,3)="a2" *{\bullet}, "a2"; "a2"-"hd" **@{-},
"h"+(0,0) *{\vdots},
"h"+(0,-5)="a3" *{\bullet}, "a3"; "a3"-"hd" **@{-},
"h"+(0,-8)="a4" *{\bullet}, "a4"; "a4"-"hd" **@{-},
"h"+(0,0) *\xycircle(5,12){-}, "h"+(0,-10) *++!UL{V_2},
"h"; "h"+"hd"+(10,0)="h",
"h"+(0,6)="a1" *{\bullet} *+!D{v_d}, "a1"; "a1"-"hd" **@{-},
"h"+(0,3)="a2" *{\bullet}, "a2"; "a2"-"hd" **@{-},
"h"+(0,0) *{\vdots},
"h"+(0,-5)="a3" *{\bullet}, "a3"; "a3"-"hd" **@{-},
"h"+(0,-8)="a4" *{\bullet}, "a4"; "a4"-"hd" **@{-},
"h"+(0,0) *\xycircle(5,12){-}, "h"+(0,-10) *++!UL{V_d},
\end{xy}\qquad
$$
\caption{Schematic diagram for the graph $\Lambda$
in Lemma~\ref{lem:minlox}}
\label{fig:minlox}
\end{figure}
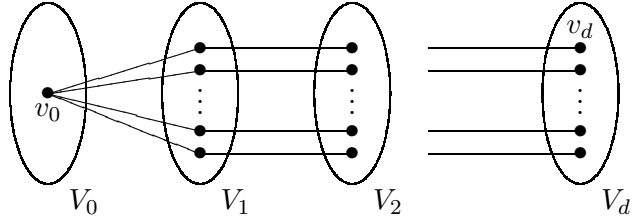

By putting $\Lambda = \bar\Gamma$ in the above lemma,
we obtain the following theorem.
\begin{theorem}\label{thm:d_TL}
Let $\Gamma$ be a biconnected graph.
If\/ $d=\diam(\bar\Gamma)\ge 3$, then
$$\minTL_{(\Gamma^e,\,d_e)}(A(\Gamma))\le \frac{2}{d-2}.
$$
\end{theorem}

From the above theorem, if $d=\diam(\bar\Gamma)$ is large,
then $\minTL_{(\Gamma^e,\,d_e)}(A(\Gamma))$ is small.
One may ask whether the converse holds, i.e.\
if $\minTL_{(\Gamma^e,\,d_e)}(A(\Gamma))$ is small, then $\diam(\bar\Gamma)$ is large.
The following theorem shows that this is not the case.

\begin{theorem}\label{thm:V_TL}
For any $n\ge 4$, there exists a biconnected graph $\Gamma_n$
such that
$$|V(\Gamma_n)|=n+2,\quad \diam(\Gamma_n)=\diam(\overline{\Gamma_n})=2,
\quad \minTL_{(\Gamma_n^e,\,d_e)}(A(\Gamma_n))\le \frac{2}{n-3}.$$
\end{theorem}

\begin{myproof}
Let  $\Lambda_n$ be the biconnected graph mentioned in Proposition~\ref{prop:22},
and let $\Gamma_n = \overline{\Lambda_n}$.
Then $\diam(\Gamma_n)=\diam(\overline{\Gamma_n})=2$;
$|V(\Gamma_n)|=n+2$;
$P_n$ is an induced subgraph of $\overline{\Gamma_n}$;
$V(P_n)$ dominates $\overline{\Gamma_n}$.
Since $\diam(P_n)=n-1\ge 3$,
there exists a loxodromic element
$g\in A(\Gamma_n)$
such that $\tau_e(g)\le \frac{2}{\diam(P_n)-2}=\frac{2}{n-3}$
(by Lemma~\ref{lem:minlox}).
Therefore
$$
\minTL_{(\Gamma_n^e,d_e)} (A(\Gamma_n)) \le \frac{2}{n-3}.
$$
\end{myproof}


\end{document}